\documentclass[sn-mathphys-num]{sn-jnl}

\usepackage{graphicx}%
\usepackage{multirow}%
\usepackage{amsmath,amssymb,amsfonts}%
\usepackage{amsthm}%
\usepackage{mathrsfs}%
\usepackage[title]{appendix}%
\usepackage{xcolor}%
\usepackage{textcomp}%
\usepackage{manyfoot}%
\usepackage{booktabs}%
\usepackage{algorithm}%
\usepackage{algorithmicx}%
\usepackage{algpseudocode}%
\usepackage{listings}%

\theoremstyle{thmstyleone}%
\newtheorem{theorem}{Theorem}
\theoremstyle{thmstyletwo}%
\newtheorem{example}{Example}%

\theoremstyle{thmstylethree}%

\begin{document}

\title[A posteriori error estimates for the exponential midpoint method  for  linear and semilinear parabolic equations]{A posteriori error estimates for the exponential midpoint method  for  linear and semilinear parabolic equations}


\author[1]{\fnm{Xianfa} \sur{Hu}}\email{zzxyhxf@163.com}

\author*[1]{\fnm{Wansheng} \sur{Wang}}\email{w.s.wang@163.com}

\author[1]{\fnm{Mengli} \sur{Mao}}\email{mmlmath@163.com}

\author[1]{\fnm{Jiliang} \sur{Cao}}\email{jiliangcao@shnu.edu.cn}

\affil[1]{\orgdiv{Department of Mathematics}, \orgname{Shanghai Normal University}, \orgaddress{\street{ No. 100 Guilin Road}, \city{Shanghai}, \postcode{200234}, \state{Shanghai}, \country{P.R.China}}}


\abstract{In this paper, the a posteriori error estimates of the exponential midpoint method for time discretization are studied  for linear and  semilinear parabolic equations. Using the exponential midpoint approximation defined by  a  continuous and  piecewise linear interpolation  of  nodal values   yields the suboptimal order estimates. Based on the property of the entire function, we introduce a continuous and  piecewise  quadratic time reconstruction of the exponential midpoint method to  derive the  optimal order estimates, and the error bounds are solely  dependent on the discretization parameters, the data of the problem and the approximation of the entire function. Several  numerical examples are implemented to illustrate the  theoretical results.}


\maketitle
\section{Introduction}
This paper is devoted to deriving   optimal order a  posteriori error estimates  for 
the exponential midpoint method in time for linear and semilinear parabolic problems
\begin{equation}\label{nonlinear problem}\left\{
\begin{array}{l}
u'(t)+Au(t)=B(t,u(t)),\quad 0\leq t\leq T,\cr\noalign{\vskip1truemm} u(0)=u^0,
\end{array}
\right.\end{equation}
{where}  the operator $A: D(A)\rightarrow H$ is positive definite, self-adjoint and linear  on a Hilbert space $( H,(\cdot,\cdot))$ with domain $D(A)$ dense in $H$, $B(t, \cdot) : D(A)\rightarrow H, \ t\in [0,T]$, and the initial value $u^0\in H$. It is well known that the reaction-diffusion equations and incompressible
Navier-Stokes equations fit into this framework \cite{Henry1981,Lunardi1995}. Problems of the form \eqref{nonlinear problem} arise frequently in  various fields such as mechanics, engineering and other applied sciences.

 The idea of exponential integrators is very old, Lawson \cite{Lawson1967} formulated explicit Runge-Kutta methods with A-stability by using the exponential functions, Friedli developed this idea in \cite{Friedli1978}, and exponential multistep methods have been considered in \cite{Lambert1972,Verwer1977}. Generally,  exponential 
 integrators provide higher accuracy and stability than non-exponential ones for stiff systems and highly oscillatory problems, and its implementation depends on the computing   the product of an exponential function with a vector, with the new methods for evaluating  it (see, e.g., \cite{Hochbruck1997,Berland2007}), exponential integrators have been received more attention \cite{Norsett1978,Cox2002,Hochbruck2005a,Krogstad2005,Moret2005,Hochbruck2010,Li2016,Mei2017,Du2019,WangBin2019,Du2021,Li2022}. It is worth mentioning
that
Hochbruck et al. \cite{Hochbruck2010} systematically presented exponential one step and multistep methods, which were based on the stiff order conditions. In this paper, we focus on exponential one step methods. The exponential Runge-Kutta methods of collocation type were proposed in   \cite{Hochbruck2005b}. However,  optimal order a posteriori estimates for exponential Runge-Kutta methods of collocation type  is still an open problem.  There are some works about the earliest significant contribution to a posteriori analysis for time or space discrete approximations for evolution problems
\cite{Johnson1990,Noch2000,Giles2002,Makri2003}. Akrivis et al.  derived the a posteriori error estimates  of evolution problems  for Crank-Nicolson method \cite{Akrivis2006}, and continuous Galerkin method \cite{Akrivis2009}. Makridakis and Nochetto \cite{Makri2006} considered  the a posteriori error estimates for discontinuous Galerkin method. A posteriori error analysis for Crank-Nicolson-Galerkin  type methods and fully  discrete finite element method  for the generalized diffusion equations with delay have been derived by Wang et al. \cite{Wang2018,Wang2020,Wang2022b}.  Unfortunately, there is no article devoting to the a posteriori error analysis for standard exponential integrators for solving parabolic differential equations. Owing to this, our study starts by deriving the a posteriori error estimates for the exponential midpoint
method. Considering the property of the exponential function, we introduce the continuous,
linear piecewise  interpolation and quadratic time reconstruction of the exponential midpoint method to obtain the suboptimal and optimal estimates, respectively.

This paper is mainly presented the a posteriori error estimates of  the exponential midpoint method for linear parabolic equations with $B(t,u(t))=f(t)$ as a given forcing function $f$. We have a brief discussion on the a posteriori error estimates
for semilinear problems.  The paper is organized as follows. Section \ref{sec2} introduces the necessary notation, and briefly
reviews  exponential Runge-Kutta methods of collocation type and exponential midpoint method. In Section \ref{sec3}, by introducing a continuous approximation $U$ in time, the residual-based a posterior error bounds, which are suboptimal order, are derived for linear parabolic problems by using the energy techniques. To derive optimal order a posteriori error estimates,  the quadratic exponential midpoint reconstruction $\hat{U}$ is introduced in Section \ref{sec4}. Then a posteriori error estimates in $L^2(0,t;V)$- and $L^{\infty}(0,t;H)$-norms are derived for the exponential midpoint method for linear problems and the difference $\hat{U}-U$ is computed in Section \ref{sec5}. Section \ref{sec6} is devoted to the extension of the a posteriori error analysis of the  exponential midpoint method from linear problems to semilinear problems. The effectiveness of the a posteriori error estimators is verified by several numerical examples in Section \ref{sec7}.  The last section is concerned with concluding remarks.

\section{The exponential midpoint method for parabolic problems}\label{sec2}

Now we consider the exponential midpoint method for solving parabolic problems \eqref{nonlinear problem}. At first,  some appropriate assumptions and  notations are needed.

\subsection{Assumptions and notations}
Let $V:=D(A^{\frac{1}{2}})$ and denote the norms in $H$ and $V$ by $|\cdot|$ and $\|\cdot\|$, $\|v\|=|A^{\frac{1}{2}}v|=(Av,v)^{\frac{1}{2}}$, respectively. We identify $H$ with its dual, and let $V^{\star}$
be the dual of $V$ ($V\in H \in V^{\star}$), and denote by $\|\cdot\|$ the dual norm on $V^{\star}$, $\|v\|_{\star}=|A^{-\frac{1}{2}}v|=(v,A^{-1}v)^{\frac{1}{2}}.$  We still denote by $(\cdot,\cdot)$ the duality pairing between $V^{\star}$ and $V$, and introduce the Lebesgue spaces $L^p(J;X)$ with the interval $J$ and Banach space $X$ ($X=H,\ V \ or \ V^{\star}$)  and $L^p$-norm of $\|u(t)\|_{X}$ , $1\leq p < \infty$ are
\begin{equation*}
\|u(t)\|_{L^p(J;X)}=\Big(\int_{J} \|u(t)\|_{X}^{p}dt \Big)^{1/p},
\end{equation*}
is finite. The space $L^{\infty}(J;X)$    with interval $J$ and Banach space $X$, and the $L^{\infty}$-norm of $\|u(t)\|_{X}$ are
\begin{equation*}
\|u\|_{L^{\infty}(J;X)}:= \text{ess} \sup\limits_{t \in J} \|u(t)\|_{X}.
\end{equation*}
For the sake of convenience, we write  $L^p(0,t;X)$ for  $L^p((0,t);X)$. Suppose that  the Poincar\'{e} inequality holds:
\begin{equation}
|u|\leq \frac{1}{\sqrt{\lambda_1}}\|u\|, \quad \forall u \in V,
\end{equation}
here $\lambda_1>0$ is the first eigenvalue of the operator $A$.

\subsection{Exponential Runge-Kutta methods of collocation type} Let $ 0= t^0 < t^1 < \cdots < t^{N} = T$ be a partition of $[0, T]$, $I_n:=(t^{n-1},t^n]$, and $k_n:=t^n-t^{n-1}$. It is well known that  the main idea of exponential integrator is
concerned with  the  Volterra integral equation
 \begin{equation}\label{variation of constants}
\displaystyle u(t^{n})=e^{-k_nA}u(t^{n-1})+\int_{t^{n-1}}^{t^n} e^{-(t^n-\tau)A}B(\tau,u(\tau))d\tau,
 \end{equation}
also called \emph{variation-of-constants} formula for solving  semilinear parabolic differential equations (here, the nonlinear term is $B(\tau,u(\tau))$).

An $s$-stage Exponential Runge--Kutta (ERK) method for \eqref{nonlinear problem} is defined as
\begin{equation}\label{ERK}
\left\{
\begin{array}{l}
\displaystyle U^{n-1,i}=e^{{-c}_ik_nA}U^{n-1}+k_n\sum\limits_{j=1}^sa_{ij}(-k_nA)B(t^{n-1,j},U^{n-1,j}), \cr\noalign{\vskip4truemm}
\displaystyle U^{n}=e^{-k_nA}U^{n-1}+k_n\sum\limits_{i=1}^sb_{i}(-k_nA)B(t^{n-1,i},U^{n-1,i});
\end{array}
\right.
\end{equation}
{where}  the coefficients $a_{ij}(-k_nA)$, $b_i(-k_nA)$
 are exponential functions of $k_nA$ for ${i,j=1,\ldots s}$, 
 $t^{n-1,i}=t^{n-1}+c_ik_n$, $U^{n-1,i}\approx u(t^{n-1}+c_ik_n)$ for $i=1,\ldots,s$. The method (\ref{ERK}) can  be  represented  by the  Butcher tableau
\begin{equation}\label{Tableau-REK}
\begin{aligned}
\begin{tabular}{c|ccc}
$c_1$&$a_{11}(-k_nA)$&$\cdots$ &$a_{1s}(-k_nA)$\\
 $\raisebox{-1.3ex}[1.0pt]{$\vdots$}$&$\raisebox{-1.3ex}[1.0pt]{$\vdots$}$& & $\raisebox{-1.3ex}[1.0pt]{$\vdots$}$\\
  $\raisebox{-1.3ex}[1.0pt]{$c_s$}$ &  $\raisebox{-1.3ex}[1.0pt]{$a_{s1}(-k_nA)$}$ & $\raisebox{-1.3ex}[1.0pt]{$\cdots$}$&   $\raisebox{-1.3ex}[1.0pt]{$a_{ss}(-k_nA)$}$\\[6pt]
 \hline
&$\raisebox{-1.3ex}[1.0pt]{$b_1(-k_nA)$}$&\raisebox{-1.3ex}[1.0pt]{$\cdots$} &  $\raisebox{-1.3ex}[1.0pt]{$b_s(-k_nA)$}$\\
\end{tabular}.
\end{aligned}
\end{equation}

For $0 < c_1 < \cdots <  c_s \leq 1$,  ERK methods of collocation  type can be viewed as  using the interpolation polynomial $P_{s-1}$ of degree $s-1$   about  collocation nodes  $\{c_i\}_{i=1}^s$ to  approximate the nonlinear term $B(\tau,u(\tau))$ in the Volterra integral equation (\ref{variation of constants}). Hence,  the  coefficients $a_{ij}(-k_nA)$ and $b_i(-k_nA)$ are defined as follows:
\begin{equation}\label{ERK weight}
a_{ij}(-k_nA)=\int_0^{c_i} e^{-(c_i-\tau)k_nA}L_j(\tau) d\tau, \quad  b_i(-k_nA)=\int_0^1 e^{-(1-\tau)k_nA}L_i(\tau)d\tau,
\end{equation}
for $i,j=1,\ldots,s$, where $L_1(\tau), \cdots, L_s(\tau) $ are the Lagrange interpolation polynomials
\begin{equation}\label{lagrange interpolation}
L_i(\tau)=\prod\limits_{m\neq i}\frac{\tau-c_m}{c_i-c_m}.
\end{equation}
 Let $q_1$ and $q_2$ be the largest integers such that
\begin{equation}\label{order conditions}
\begin{aligned}
&\sum\limits_{i=1}^s b_i(-k_nA)\frac{c_i^{k-1}}{(k-1)!}=\varphi_k(-k_nA), \qquad  \quad k=1,\ldots,q_1,\\
&\sum\limits_{j=1}^s a_{ij}(-k_nA)\frac{c_j^{k-1}}{(k-1)!}=\varphi_k(-c_ik_nA)c_i^k, \quad k=1,\ldots,q_2,\quad i=1,\ldots,s,
\end{aligned}
\end{equation}
with the entire functions
\begin{equation}\label{entire function}
\varphi_k(-c_ik_nA)=\int_0^1 e^{-(1-\tau)c_ik_nA}\frac{\tau^{k-1}}{(k-1)!}d\tau,  \quad i=1,\ldots,s,
 \end{equation}
 where $q_1$ is  the order of the internal stages  and $q_2$ is  the order of the  update. It is clear that if we consider  the limit $A\rightarrow \bf{0}$, the formulas (\ref{ERK weight}) and (\ref{order conditions}) reduce to \begin{equation}\label{RK coefficients}
a_{ij}=\int_0^{c_i} L_j(\tau) d\tau, \quad  b_i=\int_0^1 L_i(\tau)d\tau,\quad i,j=1,\ldots,s,
\end{equation}
and
\begin{equation}\label{RK order conditions}
\begin{aligned}
&\sum\limits_{i=1}^s b_i\frac{c_i^{k-1}}{(k-1)!}=\frac{1}{k!}, \qquad  k=1,\ldots,q_1,\\
&\sum\limits_{j=1}^s a_{ij}\frac{c_j^{k-1}}{(k-1)!}=\frac{c_i^k}{k!}, \quad k=1,\ldots,q_2, \quad i=1,\ldots,s.
\end{aligned}
\end{equation}
This implies that the ERK method of collocation type reduces to a RK method of collocation type when $A\rightarrow \bf{0}$, therefore the limit method is called  the underlying RK method. Hochbruck et al. \cite{Hochbruck2005b} have illustrated that an $s$-stage ERK  method of collocation type converges with order $\min(s+1,p)$ for linear parabolic problems, where $p$ denotes the classical (non-stiff) order of the underlying RK method. For semilinear parabolic problems, higher order ERK methods were formulated  like  the exponential Gauss methods in some special cases (sufficient temporal and spatial smoothness).

\subsection{The  exponential midpoint method for parabolic problems}\label{EMR} In this paper, we consider the exponential midpoint method, the simplest ERK method,  for solving  \eqref{nonlinear problem} (see \cite{Hochbruck2010}):
 \begin{equation}\label{IMERK12}
\left\{
\begin{array}{l}
\displaystyle
U^{n-\frac{1}{2}}=e^{-\frac{1}{2}k_nA}U^{n-1}+\frac{k_n}{2}\varphi_1\big(-\frac{1}{2}k_nA\big)B(t^{n-\frac{1}{2}},U^{n-\frac{1}{2}}), \cr\noalign{\vskip4truemm}
\displaystyle U^{n}=e^{-k_nA}U^{n-1}+k_n\varphi_1(-k_nA)B(t^{n-\frac{1}{2}},U^{n-\frac{1}{2}}),
\end{array}
\right.
\end{equation}
    which can be denoted by the Butcher tableau
 \begin{equation}\label{Tableau-ERK12}
\begin{aligned}
\begin{tabular}{c|c}
 $\frac{1}{2}$&$\frac{1}{2}\varphi_1(-\frac{1}{2}k_nA)$\\[4pt]
 \hline
& $\raisebox{-1.3ex}[1.0pt]{$\varphi_1(-k_nA)$}$\\
\end{tabular}
\end{aligned}
\end{equation}
   where $\varphi_1(-k_nA)$ and  $\varphi_1(-\frac{1}{2}k_nA)$ are defined by \eqref{entire function}.


\section{A posteriori error estimates for linear problems} \label{sec3}
We consider  the  linear parabolic differential equations (LPDEs):
\begin{equation}\label{linear problem}\left\{
\begin{array}{l}
u'(t)+Au(t)=f(t),\quad 0\leq t \leq T,\cr\noalign{\vskip1truemm} u(0)=u^0,
\end{array}
\right.\end{equation}
where the forcing term $f$ is a sufficiently smooth function.
It is worth noting that if we apply an  ERK method of collocation type  to LPDEs, the method reduces to the exponential quadrature rule
\begin{equation}\label{ERK-update}
U^{n}=e^{-k_nA}U^{n-1}+k_n\sum\limits_{i=1}^s b_{i}(-k_nA)f(t^{n-1,i}),
\end{equation}
where  the weights are 
\begin{equation}\label{the quadrature weight}
 b_i(-k_nA)=\int_0^1 e^{-(1-\tau)k_nA}L_i(\tau)d\tau,\quad  i=1,\ldots,s.
\end{equation}

\subsection{Exponential midpoint rule} For given $\{v^n\}_{n=0}^N$, we introduce the notation
\begin{equation*}
\bar{\partial} v^n:=\frac{v^n-v^{n-1}}{k_n},\quad  v^{n-\frac{1}{2}}:=\frac{v^n+v^{n-1}}{2}, \quad  n=1,\ldots,N.
\end{equation*}
For \eqref{linear problem}, the  method \eqref{IMERK12} reduces to the  exponential midpoint rule:
 \begin{equation}\label{exponentail M}
U^{n}=e^{-k_nA}U^{n-1}+k_n\varphi_1(-k_nA)f(t^{n-\frac{1}{2}}),\quad n=1,\ldots,N.
\end{equation}
Combining the recurrence relation
\begin{equation}\label{phifunction}
\varphi_{k+1}(z)=\frac{\varphi_k(z)-\varphi_k(0)}{z},\quad \varphi_0(z)=e^z, \quad k\geq1,
\end{equation}
and $\varphi_k(0)=1/k!$, we have
\begin{equation*}
  e^{-k_nA}=I-k_n\varphi_1(-k_nA)A.
\end{equation*}
 Hence,   $U^n$ can be rewritten as
\begin{equation}
U^{n}=U^{n-1}+k_n\varphi_1(-k_nA)\big(f(t^{n-\frac{1}{2}})-AU^{n-1}\big),
\end{equation}
which implies that
\begin{equation}\label{another form of ERK-M}
\bar{\partial}U^n+\varphi_1(-k_nA)AU^{n-1}=\varphi_1(-k_nA)f(t^{n-\frac{1}{2}}), \quad n=1,\ldots,N,
\end{equation}
with $U^0:=u^0$. In \cite{Hochbruck2005b}, Hochbruck et al. have presented that if  $f^{(s+1)}\in L^{1}(0,T;X)$, and the ERK
method satisfies the conditions
\begin{equation}\label{additional conditions}
\sum\limits_{i=1}^sb_i(0)c_{i}^{s}=\frac{1}{s+1},
\end{equation}
then the  error bound  $\|u(t^n)-U^{n}\|\leq Ch^{s+1}$ holds on $0\leq t^n \leq T$. As a consequence, it is easy to verify that the exponential midpoint rule is of second order.

The  exponential midpoint  approximation $U : \ [0,T] \rightarrow D(A)$ to $u$ is defined by using the linear interpolation between the nodal values $U^{n-1}$ and $U^n$,
\begin{equation}\label{linear interpolation}
U(t)=U^{n-1}+(t-t^{n-1})\frac{U^n-U^{n-1}}{k_n}, \quad  t \in I_n.
\end{equation}
Let $R(t)\in H$ be the residual of $U$,
\begin{equation}
R(t)=U^{\prime}(t)+AU(t)-f(t),\quad t \in I_n,
\end{equation}
which can be considered as the approximation solution $U$ misses the exact solution of \eqref{linear problem}.  A direct calculation yields
\begin{equation}
U^{\prime}(t)+AU(t)=\bar{\partial}U^n+AU^{n-1}+(t-t^{n-1})A\bar{\partial}U^n, \quad t \in I_n.
\end{equation}
It follows from \eqref{another form of ERK-M} that
\begin{equation*}
U^{\prime}(t)+AU(t)=(I-\varphi_1(-k_nA))AU^{n-1}+\varphi_1(-k_nA)f(t^{n-\frac{1}{2}})+(t-t^{n-1})A\bar{\partial}U^n, \quad t \in I_n.
\end{equation*}
Therefore, the residual $R(t)$ of $U$ is
\begin{equation*}
R(t)=\varphi_1(-k_nA)f(t^{n-\frac{1}{2}})-f(t)+(I-\varphi_1(-k_nA))AU^{n-1}+(t-t^{n-1})A\bar{\partial}U^n, \quad t \in I_n.
\end{equation*}
Using  the recurrence relation \eqref{phifunction} and
 \begin{equation}\label{series of phi}
 \varphi_{k+1}(z)=\sum\limits_{j=0}^{\infty} \frac{z^j}{(j+k+1)!},
 \end{equation}
 we get 
 \begin{equation*}
 I-\varphi_1(-k_nA)=\mathcal{O}(k_n),\quad n=1,\ldots, N.
 \end{equation*}
 For $ t \in I_n$,  the residual $R(t)$ also can be rewritten in the following form
\begin{equation}\label{residual of EMR}
R(t)=(\varphi_1(-k_nA)-I)(f(t^{n-\frac{1}{2}})-AU^{n-1})+f(t^{n-\frac{1}{2}})-f(t)+(t-t^{n-1})A\bar{\partial}U^n.
\end{equation}

Let the error $e:=u-U$. According to
\begin{equation*}\left\{
\begin{array}{l}
u'(t)+Au(t)=f(t),\cr\noalign{\vskip4truemm}
{U}^{\prime}(t)+AU(t)=f(t)+R(t), 
\end{array}
\right.
\end{equation*}
then the error $e(t)$ satisfies
\begin{equation}\label{error of U}
e^{\prime}(t)+Ae(t)=-R(t).
\end{equation}
Taking in \eqref{error of U} the inner product with $e(t)$ leads to
\begin{equation}
(e^{\prime}(t),e(t))+(Ae(t),e(t))=-(R(t),e(t)),
\end{equation}
then
\begin{equation}\label{eq4.4}
\frac{1}{2}\frac{d}{dt}|e(t)|^2+\|e(t)\|^2=-(R(t),e(t)).
\end{equation}
\subsection{Maximal norm estimate} Combining the Poincar\'{e} inequality with  the Cauchy-Schwarz inequality, from (\ref{eq4.4}) we have
\begin{equation*}
\frac{1}{2}\frac{d}{dt}|e(t)|^2+\lambda_1|e(t)|^2 \leq \frac{\lambda_1}{2}|e(t)|^2+\frac{1}{2\lambda_1}|R(t)|^2.
\end{equation*}
It is easy to obtain that
\begin{equation}\label{u estimate}
\frac{d}{dt}(e^{\lambda_1t}|e(t)|^2)\leq\frac{e^{\lambda_1t}}{\lambda_1}|R(t)|^2.
\end{equation}
Integrating \eqref{u estimate} from $0$ to $t\in(0,T]$ yields
\begin{equation}
\begin{aligned}
|e(t)|^2&\leq e^{-\lambda_1t}|e(0)|^2+e^{-\lambda_1t}\int_{0}^t \frac{e^{\lambda_1s}}{\lambda_1}|R(s)|^2ds\cr\noalign{\vskip4truemm}
&\leq e^{-\lambda_1t}|e(0)|^2+\frac{1-e^{-\lambda_1t}}{\lambda_1^2}\max\limits_{s\in[0,T]}|R(s)|^2.
\end{aligned}
\end{equation}
Hence, we have
\begin{equation}
\max\limits_{s\in[0,T]} |e(s)|^2 \leq |e(0)|^2+\frac{1}{\lambda_1^2}\max\limits_{s\in[0,T]}|R(s)|^2.
\end{equation}
This indicates that the error of  $u(t)-U(t)$  is  bounded by the initial error $|e(0)|$ and the residual $R(s), \ s\in[0,T]$. Under the assumption of  $U^0=u^0$, we deduce that
\begin{equation}
\max\limits_{s\in[0,T]} |e(s)|^2 \leq \frac{1}{\lambda_1^2} \max\limits_{s\in[0,T]} |R(s)|^2.
\end{equation}

\subsection{${L^2(0,t;V)}$-estimate} We have presented  the error estimate with maximal norm in the previous subsection,  the following theorem will show the error estimate in the $L^2(0,t;V)$-norm.
\begin{theorem}
Let $u(t)$ be the exact solution of \eqref{linear problem}, and  $U(t)$ is the exponential midpoint approximation  to $u(t)$ defined by \eqref{linear interpolation}. Denoting $R(t)$ as the residual of $U(t)$, then the error
$e(t)=u(t)-U(t)$ satisfies
\begin{equation}
|e(t)|^2+\int_0^t \|e(s)\|^2ds\leq \int_0^t \|R(s)\|_{\star}^2ds.
\end{equation}
\end{theorem}
\begin{proof}
In view of \eqref{eq4.4}, using the Cauchy-Schwarz inequality and the Young inequality,   we get
\begin{equation*}
\frac{d}{dt}
|e(t)|^2+\|e(t)\|^2\leq\|R(t)\|_{\star}^2.
\end{equation*}
Under the assumption of $U^0=u^0$, we integrate the above inequality from $0$ to $t \leq T$ and obtain
\begin{equation*}
|e(t)|^2+\int_0^t \|e(s)\|^2ds \leq \int_0^t \|R(s)\|_{\star}^2ds.
\end{equation*}
The proof is completed.
\end{proof}

These results reveal that the error $e(t)=u(t)-U(t)$ can be bounded by the residual $|R(t)|$. It is obvious that  the a posteriori quantity $R(t)$  has order one. In fact, when we apply  the exponential midpoint rule to a scalar ordinary differential equation (o.d.e.)  $u^{\prime}(t)=f(t)$, the exponential midpoint rule reduces to the mid-rectangle formula, i.e.,
\begin{equation*}
R(t)=f(t^{n-\frac{1}{2}})-f(t).
\end{equation*}
In this case, obviously, $R(t)$ has order one.  However, the exponential midpoint rule has second-order accuracy, and the residual $R(t)$ is the suboptimal order.  Applying the energy techniques to the error equation leads to the suboptimal bounds.

\section{Exponential midpoint  reconstruction}\label{sec4}
Based on the  property of the exponential function,  a continuous piecewise quadratic polynomial in time $\hat{U}:[0,T]\rightarrow H$ is defined as follows, which 
provides the second-order residual. To begin with, we denote $\psi(t)$ as the linear interpolation of $\varphi_1(-k_nA)f(t)$ at the nodes $t^{n-1}$ and $t^{n-\frac{1}{2}}$,
\begin{equation}\label{new interpolation of f}
\psi(t):=\varphi_1(-k_nA)f(t^{n-\frac{1}{2}})+\frac{2}{k_n}(t-t^{n-\frac{1}{2}})\varphi_1(-k_nA)\big[f(t^{n-\frac{1}{2}})-f(t^{n-1})\big], \quad t \in I_n,
\end{equation}
and introduce  a piecewise quadratic polynomial $\Psi(t)=\int_{t^{n-1}}^{t} \psi(s)ds,\ t \in I_n,$ i.e.,
\begin{equation}
\Psi(t)=(t-t^{n-1})\varphi_1(-k_nA)f(t^{n-\frac{1}{2}})-\frac{1}{k_n}(t-t^{n-1})(t^n-t)\varphi_1(-k_nA)\big[f(t^{n-\frac{1}{2}})-f(t^{n-1})\big],
\end{equation}
which satisfies that
\begin{equation}
\Psi(t^{n-1})=0,\ \Psi(t^n)=k_n\varphi_1(-k_nA)f(t^{n-\frac{1}{2}})=\int_{I_n} \varphi_1(-k_nA)f(t^{n-\frac{1}{2}})ds.
\end{equation}
For any $t \in I_n$, the exponential midpoint reconstruction $\hat{U}$ of $U$  is defined by
\begin{equation}\label{new reconstruction}
\begin{aligned}
\hat{U}(t)&=U^{n-1}-\int_{t^{n-1}}^t \varphi_1(-k_nA) AU(s)ds+\Psi(t)+(I-\varphi_1(-k_nA))\int_{t^{n-1}}^t U^{\prime}(s)ds\cr\noalign{\vskip4truemm}
&=U^{n-1}-\int_{t^{n-1}}^t \varphi_1(-k_nA)AU(s)ds+\Psi(t)+(I-\varphi_1(-k_nA))(U(t)-U^{n-1}).
\end{aligned}
\end{equation}
The above formula can be viewed as integrating $-\varphi_1(k_nA)AU(t)+\psi(t) +(I-\varphi_1(-k_nA))U^{\prime}(t)$ from $t^{n-1}$ to $t$. Obviously, 
the derivative of $\hat{U}(t)$  in time is 
\begin{equation}
\hat{U}^{\prime}(t)+\varphi_1(-k_nA)AU(t)=\psi(t)+(I-\varphi_1(-k_nA))U^{\prime}(t),   \quad \forall t \in I_n.
\end{equation}
In view of \eqref{new reconstruction}, we take the  following approximation to evaluate the integral
\begin{equation}\label{approximation1}
\begin{aligned}
\hat{U}(t)=&U^{n-1}-(t-t^{n-1})\varphi_1(-k_nA)A\big(c_1U^{n-1}+c_2U^{n}\big)+\Psi(t)\cr\noalign{\vskip4truemm}
&+[I-\varphi_1(-k_nA)](U(t)-U^{n-1}) \quad  \forall t \in I_n,
\end{aligned}
\end{equation}
where $c_1=-k_n^{-1}(\varphi_1(-k_nA)A)^{-1}[I-\varphi_1(-k_nA)-k_n\varphi_1(-k_nA)A]$, and
 $c_2=k_n^{-1}$ $\cdot(\varphi_1(-k_nA)A)^{-1}[I-\varphi_1(-k_nA)]$, then $\hat{U}$ is
\begin{equation}\label{approximation2}
\begin{aligned}
\hat{U}(t)&=U^{n-1}+\frac{(t-t^{n-1})}{k_n}\big[(I-\varphi_1(-k_nA)-k_n\varphi_1(-k_nA)A)U^{n-1}-(I\cr\noalign{\vskip4truemm}
&\quad -\varphi_1(-k_nA)) U^{n}\big]+\Psi(t)+(I-\varphi_1(-k_nA))(U(t)-U^{n-1})\cr\noalign{\vskip4truemm}
&=U^{n-1}+\frac{(t-t^{n-1})}{k_n}\big[(I-\varphi_1(-k_nA)-k_n\varphi_1(-k_nA)A)U^{n-1}-(I\cr\noalign{\vskip3truemm}
&\quad -\varphi_1(-k_nA)) U^{n}\big]+\Psi(t)+\frac{(t-t^{n-1})}{k_n}(I-\varphi_1(-k_nA))(U^{n}-U^{n-1})\cr\noalign{\vskip3truemm}
&=U^{n-1}-(t-t^{n-1})\varphi_1(-k_nA)AU^{n-1}+\Psi(t).
\end{aligned}
\end{equation}
In fact, the exponential midpoint rule  can be viewed as using the $k_n\varphi_1(-k_nA)AU^{n-1}$ to evaluate the
integral, which  is of first order. This allows us to evaluate the integral by using the first-order approximation, and it is easy to verify   that   if  we use \eqref{approximation1} to  evaluate the
integral, then the approximation has order one as well. Moreover, we notice that $\hat{U}(t^{n-1})=U^{n-1}$  and
\begin{equation}
\hat{U}(t^n)=U^{n-1}-k_n\varphi_1(-k_nA)AU^{n-1}+k_n\varphi_1(-k_nA)f(t^{n-\frac{1}{2}})=U^n.
\end{equation}
Therefore, $\hat{U}$ and $U$ coincide at the nodes $t^0,\cdots,t^{N}$; especially, $\hat{U}:[0,T]\rightarrow H$ is
continuous.

In view of \eqref{new interpolation of f} and \eqref{new reconstruction}, we have
\begin{equation}\label{the derivative}
\begin{aligned}
\hat{U}^{\prime}(t)+\varphi_1(-k_nA)AU(t)=&\varphi_1(-k_nA)\Big(f(t^{n-\frac{1}{2}})+\frac{2}{k_n}(t-t^{n-\frac{1}{2}})\big[f(t^{n-\frac{1}{2}})\cr\noalign{\vskip3truemm}
&-f(t^{n-1})\big]\Big)+(I-\varphi_1(-k_nA))U^{\prime}(t), \quad t\in I_n.
\end{aligned}
\end{equation}
Denoting  $\hat{R}(t)$ as the residual of $\hat{U}$, then
\begin{equation}
\hat{R}(t)=\hat{U}^{\prime}(t)+A\hat{U}(t)-f(t), \quad t \in I_n.
\end{equation}
Using the formula \eqref{the derivative} leads to
\begin{equation*}
\begin{aligned}
\hat{R}(t)=&A\hat{U}(t)-\varphi_1(-k_nA)AU(t)+\varphi_1(-k_nA)\Big(f(t^{n-\frac{1}{2}})+\frac{2}{k_n}(t-t^{n-\frac{1}{2}})\big[f(t^{n-\frac{1}{2}})\cr\noalign{\vskip4truemm}
&-f(t^{n-1})\big]\Big)-f(t)+(I-\varphi_1(-k_nA))U^{\prime}(t), \quad t \in I_n.
\end{aligned}
\end{equation*}
For $t \in I_n$, the residual $\hat{R}(t)$ also can be rewritten as
\begin{equation*}
\begin{aligned}
\hat{R}(t)&=A(\hat{U}(t)-U(t))+(I-\varphi_1(-k_nA))AU(t)+(I-\varphi_1(-k_nA))U^{\prime}(t)\cr\noalign{\vskip3truemm}
&\quad +\varphi_1(-k_nA)\Big(f(t^{n-\frac{1}{2}})+\frac{2}{k_n}(t-t^{n-\frac{1}{2}})\big[f(t^{n-\frac{1}{2}})-f(t^{n-1})\big]\Big)-f(t)\cr\noalign{\vskip3truemm}
&=A(\hat{U}(t)-U(t))+(I-\varphi_1(-k_nA))\big(U^{\prime}(t)+AU(t)\big)-(I-\varphi_1(-k_nA))f(t)\cr\noalign{\vskip3truemm}
&\quad +\varphi_1(-k_nA)\Big(f(t^{n-\frac{1}{2}})+\frac{2}{k_n}(t-t^{n-\frac{1}{2}})\big[f(t^{n-\frac{1}{2}})-f(t^{n-1})\big]-f(t)\Big)\cr\noalign{\vskip3truemm}
&=A(\hat{U}(t)-U(t))+(I-\varphi_1(-k_nA))\big(U^{\prime}(t)+AU(t)-f(t)\big)\cr\noalign{\vskip3truemm}
&\quad +\varphi_1(-k_nA)\Big(f(t^{n-\frac{1}{2}}) +\frac{2}{k_n}(t-t^{n-\frac{1}{2}})\big[f(t^{n-\frac{1}{2}})-f(t^{n-1})\big]-f(t)\Big).
\end{aligned}
\end{equation*}

In what follows, it will be proved that  the  a posteriori quantity $\hat{R}(t)$ has order two; compare with \eqref{residual of EMR}.

\section{Optimal order a posteriori error estimates for linear problems} \label{sec5}
In this section, using the energy techniques, we derive  optimal order a  posteriori error bounds for the exponential midpoint reconstruction \eqref{new reconstruction}. Let the errors $e$ and $\hat{e}$ as $e:=u-U$ and $\hat{e}:=u-\hat{U}$,  $u$ and $\hat{U}$ satisfy that
\begin{equation}\left\{
\begin{array}{l}
u'(t)+Au(t)=f(t),\cr\noalign{\vskip2truemm}
\hat{U}^{\prime}(t)+\varphi_1(-k_nA)AU(t)=\psi(t)+(I-\varphi_1(-k_nA)) U^{\prime}(t),
\end{array}
\right.
\end{equation}
then  we obtain
\begin{equation}\label{error2}
\begin{aligned}
\hat{e}^{\prime}(t) +Ae(t)&=f(t)-\psi(t)-(I-\varphi_1(-k_nA))(U^{\prime}(t)+AU(t))\cr\noalign{\vskip2truemm}
&=\varphi_1(-k_nA)R_f(t)+(\varphi_1(-k_nA)-I)R(t),
\end{aligned}
\end{equation}
where $R_f(t)$ and $R(t)$ are defined by
\begin{equation}\label{estimate1}
R_f(t)=f(t)-\Big\{f(t^{n-\frac{1}{2}})+\frac{2}{k_n}(t-t^{n-\frac{1}{2}})\big[f(t^{n-\frac{1}{2}})-f(t^{n-1})\big]\Big\}, \quad t\in I_n,
\end{equation}
and
\begin{equation}\label{estimate2}
R(t)=U^{\prime}(t)+AU(t)-f(t), \quad t\in I_n.
\end{equation}

Taking in \eqref{error2} the inner product with $\hat{e}(t)$ leads to
\begin{equation*}
(\hat{e}^{\prime}(t),\hat{e}(t))+(Ae(t),\hat{e}(t))=(\varphi_1(-k_nA)R_f(t),\hat{e}(t))+((\varphi_1(-k_nA)-I)R(t),\hat
{e}(t)),
\end{equation*}
combining
\begin{equation}
(Ae(t),\hat{e}(t))=\frac{1}{2}\big(\|e(t)\|^2+\|\hat{e}(t)\|^2-\|\hat{e}(t)-e(t)\|^2\big)
\end{equation}
with
\begin{equation}
\hat{e}(t)-e(t)=U(t)-\hat{U}(t),
\end{equation}
we deduce that
\begin{equation*}
\begin{aligned}
\frac{d}{dt}|\hat{e}(t)|^2+\|e(t)\|^2+\|\hat{e}(t)\|^2=&\|\hat{U}(t)-U(t)\|^2+2(\varphi_1(-k_nA)R_f(t),\hat{e}(t))\cr\noalign{\vskip4truemm}
&+2((\varphi_1(-k_nA)-I)R(t),\hat{e}(t)).
\end{aligned}
\end{equation*}

\subsection{${L^2(0,T;V)}$-estimate}
We show the $L^2(0,T;V)$-norm a posteriori error estimate for the exponential midpoint rule by the following theorem.
\begin{theorem}
Let $u(t)$ be the exact solution of \eqref{linear problem}, $U(t)$ is the exponential midpoint approximation to $u(t)$ defined by \eqref{linear interpolation},   $\hat{U}$ is the corresponding reconstruction of
$U$ defined by \eqref{new reconstruction}, and denote $e=u-U$, $\hat{e}=u-\hat{U}$ respectively, then a posteriori error estimate holds for $t\in I_n$:
\begin{equation}\label{reconstruction estimate1}
\begin{aligned}
\frac{2}{5}\int_0^t&\|\hat{U}(s)-U(s)\|^2ds\leq |\hat{e}(t)|^2+\int_0^t \Big(\|e(s)\|^2+\frac{2}{3}\|\hat{e}(s)\|^2\Big)ds \cr\noalign{\vskip4truemm}
\leq& \int_0^t\|\hat{U}(s)-U(s)\|^2ds\\
& + 6\sum\limits_{i=1}^{n-1}\int_{t^{i-1}}^{t^{i}} \Big(\|\varphi_1(-k_iA)R_f(s)\|_{\star}^2+\|(\varphi_1(-k_iA)-I)R(s)\|_{\star}^2\Big)ds\\
&+6\int_{t^{n-1}}^{t} \Big(\|\varphi_1(-k_nA)R_f(s)\|_{\star}^2 +\|(\varphi_1(-k_nA)-I)R(s)\|_{\star}^2\Big)ds,
\end{aligned}
\end{equation}
where $R_f(s)$ and $R(s)$ are defined by \eqref{estimate1} and \eqref{estimate2}, respectively.
\end{theorem}
\begin{proof}
Using the following inequalities
\begin{equation}
\begin{aligned}
((\varphi_1(-k_nA)-I)R(t),\hat{e}(t))&\leq 3\|(I-\varphi_1(-k_nA))R(t)\|_{\star}^2+\frac{1}{12}\|\hat{e}(t)\|^2,\\
(\varphi_1(-k_nA)R_f(t),\hat{e}(t))&\leq 3\|\varphi_1(-k_nA)R_f(t)\|_{\star}^2+\frac{1}{12}\|\hat{e}(t)\|^2,
\end{aligned}
\end{equation}
we have
\begin{equation}
\begin{aligned}
\frac{d}{dt}|\hat{e}(t)|^2+\|e(t)\|^2+\frac{2}{3}\|\hat{e}(t)\|^2\leq&\|\hat{U}(t)-U(t)\|^2+6\Big(\|\varphi_1(-k_nA)R_f(t)\|_{\star}^2\cr\noalign{\vskip3truemm}
&+\|(\varphi_1(-k_nA)-I)R(t)\|_{\star}^2\Big).
\end{aligned}
\end{equation}
Integrating the above inequality from $0$ to $t\in I_{n}$ leads to
\begin{equation}
\begin{aligned}
|\hat{e}(t)|^2&+\int_0^t \big(\|e(s)\|^2+\frac{2}{3}\|\hat{e}(s)\|^2\big)ds \cr\noalign{\vskip3truemm}
\leq& \int_0^t\|\hat{U}(s)-U(s)\|^2ds + 6\sum\limits_{i=1}^{n-1}\int_{t^{i-1}}^{t^{i}} \Big(\|\varphi_1(-k_iA)R_f(s)\|_{\star}^2 \cr\noalign{\vskip3truemm}
& +\|(\varphi_1(-k_iA)-I)R(s)\|_{\star}^2\Big)ds\\
&+ 6\int_{t^{n-1}}^{t}\Big(\|\varphi_1(-k_nA)R_f(s)\|_{\star}^2 +\|(\varphi_1(-k_nA)-I)R(s)\|_{\star}^2\Big)ds.
\end{aligned}
\end{equation}

 Next the lower bound will be analyzed. Using the triangle inequality
\begin{equation*}
\|\hat{U}(s)-U(s)\|\leq \|e(s)\|+\|\hat{e}(s)\|,
\end{equation*}
we get
\begin{equation}\label{inequality01}
\|\hat{U}(s)-U(s)\|^2\leq\frac{5}{2}\big(\|e(s)\|^2+\frac{2}{3}\|\hat{e}(s)\|^2\big).
\end{equation}
Likewise, we integrate \eqref{inequality01} from $0$ to $t$ and obtain the desired lower bound
\begin{equation}
\frac{2}{5}\int_0^t \|\hat{U}(s)-U(s)\|^2ds \leq \int_0^t \big(\|e(s)\|^2+\frac{2}{3}\|\hat{e}(s)\|^2\big)ds.
\end{equation}
The proof is completed.
\end{proof}
\subsection{${L^{\infty}(0,t;H)}$-estimate}
From  \eqref{reconstruction estimate1}, the $L^{\infty}(0,t;H)$-estimate is
 shown by: for $\forall t\in I_n$,
\begin{equation}\label{max-error of recon}
\begin{aligned}
 \max\limits_{0\leq s\leq t} |\hat{e}(s)|^2 \leq& \int_0^t\|\hat{U}(s)-U(s)\|^2ds + 6\sum\limits_{i=1}^{n-1}\int_{t^{i-1}}^{t^{i}} \Big(\|\varphi_1(-k_iA)R_f(s)\|_{\star}^2\cr\noalign{\vskip3truemm}
 &+\|(\varphi_1(-k_iA)-I)R(s)\|_{\star}^2\Big)ds+6\int_{t^{n-1}}^{t} \Big(\|\varphi_1(-k_nA)R_f(s)\|_{\star}^2\cr\noalign{\vskip3truemm}
 &+\|(\varphi_1(-k_nA)-I)R(s)\|_{\star}^2\Big)ds.
 \end{aligned}
\end{equation}
If we consider $\hat{U}(t)$ and $U(t)$ coincide at the nodes $t^0,\cdots,t^{N}$, then the estimate is valid
\begin{equation}
\begin{aligned}
 \max\limits_{0\leq i\leq N} |{e}(t^{i})|^2 \leq& \int_0^{T} \|\hat{U}(s)-U(s)\|^2ds + 6\sum\limits_{n=1}^{N}\int_{t^{n-1}}^{t^n} \Big(\|\varphi_1(-k_nA)R_f(s)\|_{\star}^2\cr\noalign{\vskip3truemm}
 &+\|(\varphi_1(-k_nA)-I)R(s)\|_{\star}^2\Big)ds.
 \end{aligned}
\end{equation}

\subsection{Estimate of ${\hat{U}-U}$} In this subsection, we will estimate  and   derive the representation of $\hat{U}-U$ for the exponential midpoint rule.
Subtracting \eqref{linear interpolation} from \eqref{approximation2}, we get
\begin{equation}\label{posteriori quantity1}
\begin{aligned}
\hat{U}(t)-U(t)&=U^{n-1}-(t-t^{n-1})\varphi_1(-k_nA)AU^{n-1}+\Psi(t)-U(t)\cr\noalign{\vskip3truemm}
&=-(t-t^{n-1})\varphi_1(-k_nA)AU^{n-1}+\Psi(t)-(t-t^{n-1})\frac{U^{n}-U^{n-1}}{k_n}\cr\noalign{\vskip3truemm}
&=-\frac{1}{k_n}(t-t^{n-1})(t^n-t)\varphi_1(-k_nA)\big[f(t^{n-\frac{1}{2}})-f(t^{n-1})\big],
\end{aligned}
\end{equation}
for $t\in I_n$, which means that $\max_{t\in I_n} |\hat{U}(t)-U(t)|=\mathcal{O}(k_n^2)$. As we denote $\varepsilon_U:=\int_0^T \|\hat{U}(t)-U(t)\|^2dt$, the estimator is
\begin{equation}\label{posteriori quantity2}
\begin{aligned}
\varepsilon_U&=\int_0^T \|\hat{U}(t)-U(t)\|^2dt=\sum\limits_{n=1}^N \int_{I_n} \|\hat{U}(t)-U(t)\|^2dt\cr\noalign{\vskip3truemm}
&=\sum\limits_{n=1}^N \int_{I_n} (t-t^{n-1})^2(t^n-t)^2\Big\|\frac{1}{k_n}\varphi_1(-k_nA)\big[f(t^{n-\frac{1}{2}})-f(t^{n-1})\big]\Big\|^2dt\cr\noalign{\vskip3truemm}
&=\frac{1}{30}\sum\limits_{n=1}^N k_n^5\Big\|\frac{1}{k_n}\varphi_1(-k_nA)\Big[f(t^{n-\frac{1}{2}})-f(t^{n-1})\big]\Big\|^2.
\end{aligned}
\end{equation}
We make some remarks for this section:
\begin{enumerate}[(i)]
\item The order of the a posteriori quantities $\varphi_1(-k_nA)R_f(t)$ and $(\varphi_1(-k_nA)-I)R(t)$ will be discussed.   The entire
     functions $\varphi_k(-k_nA)$ are bounded  (see \cite{Hochbruck2010}),   then a posteriori quantities  $\varphi_1(-k_nA)R_f(t)$  and $\hat{U}(t)-U(t)$ are of second order once $f\in C^2(0,T;H)$.  As we presented in
    Section \ref{sec3}, the posteriori quantity $R(t)$ is of first order and $\varphi_1(-k_nA)-I=\mathcal{O}(k_n)$,
   whence a posteriori quantity  $(\varphi_1(-k_nA)-I)R(t)$  has order two.
     However the order of the quantities  $\varphi_1(-k_nA)R_f(t)$  and $\hat{U}(t)-U(t)$  may be reduced due to the practical computation of the entire function $\varphi_1(-k_nA)$.  We also notice that $I-\varphi_1(-k_nA)=\mathcal{O}(k_n)$ is a truncation of the series in \eqref{series of phi}, and this truncation usually  affects the order of $(\varphi_1(-k_nA)-I)R(t)$. It is true that the order of the exponential midpoint method  may reduce in $L^2(0,t;V)$-norm, therefore it is acceptable that the  $\varphi_1(-k_nA)$ affects the orders of the a posteriori quantities  $\hat{U}(t)-U(t)$, $\varphi_1(-k_nA)R_f(t)$ and $(\varphi_1(-k_nA)-I)R(t)$ via the $L^2(0,t;V)$-norm and $L^2(0,t;V^{\star})$-norm, respectively.    The theoretical results will be illustrated by several numerical examples.
\item   We have presented a posteriori error  estimates  for the exponential midpoint rule (\ref{exponentail M}) by using the reconstruction $\hat{U}$  in the  $L^2(0,t;V)$- and $L^{\infty}(0,t;H)$-norms. From \eqref{reconstruction estimate1} and
\eqref{max-error of recon}, it is very clear that the error upper bounds are dependent on $\hat{U}(t)-U(t)$, $\varphi_1(-k_nA)R_f(t)$ and $(\varphi_1(-k_nA)-I)R(t)$ via the $L^2(0,t;V)$-norm and $L^2(0,t;V^{\star})$-norm, respectively,  which are discretization parameters, the data of problem and the approximation of $\varphi_1(-k_nA)$. The a  posteriori quantities $\varphi_1(-k_nA)R_f(t)$ and $(\varphi_1(-k_nA)-I)R(t)$ are of the optimal second order in theory, and the
formula \eqref{posteriori quantity1} reveals that a posteriori quantity $U(t)-\hat{U}(t)$  has order two, whence  a posteriori estimate is of optimal (second) order.
    \end{enumerate}

    \section{Error estimates for semilinear parabolic differential equation}\label{sec6}
    In this section,  our objective is the derivation of a posteriori error estimates of the  exponential midpoint method for solving semilinear parabolic differential equations.
    Suppose that $B(t,\cdot)$ is an operator from $V$ to $V^{\star}$, and  the locally Lipschitz-continuous conditions for $B(t,\cdot)$  in a strip along the exact solution $u$ holds:
    \begin{equation}\label{assum1}
    \|B(t,v)-B(t,w)\|_{\star}\leq L\|v-w\|, \quad \forall v, w \in T_u,
    \end{equation}
where $T_u:=\{v\in V: \min_t\|u(t)-v\|\leq1\}$, and a constant $L$. For \eqref{nonlinear problem},  we assume that the following local one-sided
 Lipschitz condition holds:
    \begin{equation}\label{assum2}
 (B(t,v)-B(t,w),v-w)\leq \lambda\|v-w\|^2+\mu|v-w|^2, \quad \forall v, w \in T_u,
    \end{equation}
 around the solution $u$, uniformly in $t$, with a constant $\lambda$
 less than one and a constant $\mu$.

    \subsection{Exponential midpoint method}
 It has been shown that the ERK methods  converge at least
with their stage order, and that convergence of higher order (up to the classical order) occurs, if the problem has sufficient temporal and spatial smoothness. Here, the exponential midpoint method    \eqref{IMERK12}  satisfies the conditions \eqref{additional conditions}, then it is of second order.

   Since we use the nodal values $U^{n-1}$ and $U^{n}$ to express $U^{n-\frac{1}{2}}$,
then 
  \begin{equation}\label{internal stages}
  \begin{aligned}
  U^{n-\frac{1}{2}}&=e^{-\frac{1}{2}k_nA}U^{n-1}+\frac{k_n}{2}\varphi_1\big(-\frac{1}{2}k_nA\big)\varphi_1(-k_nA)^{-1}\frac{U^n-e^{-k_nA}U^{n-1}}{k_n}\cr\noalign{\vskip3truemm}
&=e^{-\frac{1}{2}k_nA}U^{n-1}+\frac{1}{2}\varphi_1\big(-\frac{1}{2}k_nA\big)\varphi_1(-k_nA)^{-1}(U^n-e^{-k_nA}U^{n-1}).
  \end{aligned}
  \end{equation}
Inserting \eqref{internal stages} into the update of the exponential midpoint method leads to
  \begin{equation}
  \begin{aligned}
  U^n&=e^{-k_nA}U^{n-1}+k_n\varphi_1(-k_nA)B(t^{n-\frac{1}{2}},U^{n-\frac{1}{2}})\cr\noalign{\vskip3truemm}
  &=U^{n-1}+k_n\varphi_1(-k_nA)\big(B(t^{n-\frac{1}{2}},U^{n-\frac{1}{2}})-AU^{n-1}\big).
  \end{aligned}
  \end{equation}
   We define the  exponential midpoint approximation $U: [0,T]\rightarrow D(A)$ to $u(t)$ by a  continuous piecewise linear polynomial of nodal values $U^{n-1}$ and $U^n$,
  \begin{equation}\label{EMM-L}
  U(t)=U^{n-1}+(t-t^{n-1})\bar{\partial}U^n, \quad t\in I_n.
  \end{equation}
  The residual $R(t) \in H$ is
  \begin{equation}
  \begin{aligned}
  R(t)&=U^{\prime}(t)+AU(t)-B(t,U(t))\cr\noalign{\vskip3truemm}
  &=(I-\varphi_1(-k_nA))AU^{n-1}+(t-t^{n-1})A\bar{\partial}U^n\cr\noalign{\vskip3truemm}
  &\qquad \quad +\varphi_1(-k_nA)B(t^{n-\frac{1}{2}},U^{n-\frac{1}{2}})-B(t,U(t)), \quad t\in I_n.
  \end{aligned}
  \end{equation}

   It is easy to verify that a posteriori quantity $R(t)$ is of first order and suboptimal order. Hence, similar to the linear case, the reconstruction $\hat{U}$ of $U$ is needed to recover the optimal order.

\subsection{Reconstruction}
 Before  we derive the optimal order a posteriori error estimates,  we need to define a continuous piecewise quadratic polynomial in time $\bar{U}: [0,T]\rightarrow H$ at nodal values $U^{n-1}$, $U^{n-\frac{1}{2}}$ and $U^{n}$. The aim of introducing $\bar{U}$ is simply to better analyze  the reconstruction $\hat{U}$ of $U$. Let $b:I_n\rightarrow H$ be the linear interpolation of $\varphi_1(-k_nA)B(\cdot,\bar{U}(\cdot))$ at the  nodes  $t^{n-1}$ and $t^{n-\frac{1}{2}}$,
 \begin{equation*}
 \begin{aligned}
 b(t)=&\varphi_1(-k_nA)B(t^{n-\frac{1}{2}},U^{n-\frac{1}{2}})\cr\noalign{\vskip4truemm}
 &+\frac{2}{k_n}(t-t^{n-\frac{1}{2}})\varphi_1(-k_nA)
 \big[B(t^{n-\frac{1}{2}},U^{n-\frac{1}{2}})-B(t^{n-1},U^{n-1})\big], \quad t\in I_n.
 \end{aligned}
 \end{equation*}
Recalling  \eqref{new reconstruction}, the
 exponential midpoint reconstruction  $\hat{U}$ of $U$  is defined by
 \begin{equation}\label{reconstruction ERK}
 \hat{U}(t):=U^{n-1}-\int_{t^{n-1}}^t \varphi_1(-k_nA)AU(s)ds +\int_{t^{n-1}}^t
 b(s)ds +\int_{t^{n-1}}^t (I-\varphi_1(-k_nA))U^{\prime}(s)ds,
 \end{equation}
 i.e.,
 \begin{equation}\label{reconstruction ERK appro}
 \begin{aligned}
 \hat{U}(t)=&U^{n-1}-(t-t^{n-1})\varphi_1(-k_nA)AU^{n-1}+(t-t^{n-1})\varphi_1(-k_nA)B(t^{n-\frac{1}{2}},U^{n-\frac{1}{2}})\cr\noalign{\vskip2truemm}
 &-\frac{1}{k_n}(t-t^{n-1})(t^n-t)\varphi_1(-k_nA)\big[B(t^{n-\frac{1}{2}},U^{n-\frac{1}{2}})-B(t^{n-1},U^{n-1})\big],
 \end{aligned}
 \end{equation}
for $t\in I_n$. It is obvious that $\hat{U}(t^{n-1})=U^{n-1}$ and
 \begin{equation*}
 \hat{U}(t^{n})=U^{n-1}-k_n\varphi_1(-k_nA)AU^{n-1}+k_n\varphi_1(-k_nA)B(t^{n-\frac{1}{2}},U^{n-\frac{1}{2}})=U^{n}.
 \end{equation*}
It follows  from \eqref{EMM-L} and  \eqref{reconstruction ERK appro} that
 \begin{equation*}
 \hat{U}(t)-U(t)=-(t-t^{n-1})(t^n-t)\frac{1}{k_n}\varphi_1(-k_nA)\big[B(t^{n-\frac{1}{2}},U^{n-\frac{1}{2}})-B(t^{n-1},U^{n-1})\big],
 \end{equation*}
 $t\in I_n$. In view of \eqref{reconstruction ERK}, we have
 \begin{equation}\label{residual1}
 \hat{U}^{\prime}(t)=-\varphi_1(-k_nA)AU(t)+\varphi_1(-k_nA)P_1B(t,\bar{U}(t))+(I-\varphi_1(-k_nA))U^{\prime}(t), \quad t\in I_n,
 \end{equation}
where $P_1B(t,\bar{U}(t))=B(t^{n-\frac{1}{2}},U^{n-\frac{1}{2}})+\frac{2}{k_n}(t-t^{n-\frac{1}{2}})
 [B(t^{n-\frac{1}{2}},U^{n-\frac{1}{2}})-B(t^{n-1},U^{n-1})]$.

Now, the \emph{residual} $\hat{R}(t)$ of $\hat{U}$ is
\begin{equation*}
\hat{R}(t)=\hat{U}^{\prime}(t)+A\hat{U}(t)-B(t,\hat{U}(t)), \quad  t\in I_n.
\end{equation*}
It immediately follows from \eqref{residual1} that
\begin{equation*}
\begin{aligned}
\hat{R}(t)=&A\hat{U}(t)-\varphi_1(-k_nA)AU(t)+\varphi_1(-k_nA)P_1B(t,\bar{U}(t))\cr\noalign{\vskip4truemm}
&+(I-\varphi_1(-k_nA))U^{\prime}(t) -B(t,\hat{U}(t)), \quad t\in I_n.
\end{aligned}
\end{equation*}
We can rewrite the residual $\hat{R}(t)$ as
\begin{equation}
\begin{aligned}
\hat{R}(t)&=A(\hat{U}(t)-U(t))+(I-\varphi_1(-k_nA))AU(t)+\varphi_1(-k_nA)P_1B(t,\bar{U}(t)) \cr\noalign{\vskip4truemm}
&\quad +(I-\varphi_1(-k_nA))U^{\prime}(t)-B(t,\hat{U}(t))\cr\noalign{\vskip4truemm}
&=A(\hat{U}(t)-U(t))+(I-\varphi_1(-k_nA))\big(U^{\prime}(t)+AU(t)-B(t,U(t))\big)\cr\noalign{\vskip4truemm}
&\quad +\varphi_1(-k_nA)\big(P_1B(t,\bar{U}(t))-B (t,U(t))\big)+B(t,U(t))-B(t,\hat{U}(t)),
\end{aligned}
\end{equation}
for $t\in I_n$, and a posteriori quantity $\hat{R}(t)$ has order two.

 \subsection{Error estimates} In this subsection, the a posteriori error bounds for the exponential midpoint method are derived. Let $e:=u-U$ and $\hat{e}:=u-\hat{U}$. We make the
 assumption that $\hat{U}(t)$, $U(t) \in T_{u}$ for all $t\in [0,T]$. Combining
 \eqref{nonlinear problem} and \eqref{residual1}, we give
\begin{equation*}\label{residual22}\left\{
\begin{array}{l}
\displaystyle
u'(t)+Au(t)=B(t,u(t)),\cr\noalign{\vskip4truemm}
\displaystyle \hat{U}^{\prime}(t)+\varphi_1(-k_nA)AU(t)=\varphi_1(-k_nA)P_1B(t,\bar{U}(t))+(I-\varphi_1(-k_nA)) U^{\prime}(t).
\end{array}
\right.
\end{equation*}
Therefore, the $\hat{e}(t)$ satisfies that 
\begin{equation}\label{residual33}
\begin{aligned}
\hat{e}^{\prime}(t)+Ae(t)=&B(t,u(t))-B(t,U(t))+(\varphi_1(-k_nA)-I)R(t)\cr\noalign{\vskip4truemm}
&+\varphi_1(-k_nA)\big(B(t,U(t))-P_1B(t,\bar{U}(t))\big),
\end{aligned}
\end{equation}
with
\begin{equation*}
R(t)=U^{\prime}(t)+AU(t)-B(t,U(t)), \quad t\in I_n.
\end{equation*}
Taking in \eqref{residual33} the inner product with $\hat{e}(t)$, and using that
\begin{equation*}
(Ae(t),\hat{e}(t))=\frac{1}{2}\big(\|e(t)\|^2+\|\hat{e}(t)\|^2-\|\hat{e}(t)-e(t)\|^2\big),
\end{equation*}
we have
\begin{equation}
\begin{aligned}
\frac{d}{dt}|\hat{e}(t)|^2+\|e(t)\|^2+\|&\hat{e}(t)\|^2=\|\hat{U}(t)-U(t)\|^2+2\big(B(t,u(t))-B(t,U(t)),\hat{e}(t)\big)\cr\noalign{\vskip4truemm}
&+2\big(\varphi_1(-k_nA)R_b(t),\hat{e}(t)\big)+2\big((\varphi_1(-k_nA)-I)R(t),\hat{e}(t)\big),
\end{aligned}
\end{equation}
with $R_b(t)=B(t,U(t))-P_1B(t,\bar{U}(t))$.
Under the assumptions of \eqref{assum1} and \eqref{assum2}, we respectively obtain
\begin{equation*}
\begin{aligned}
2(B(t,u(t))-B(t,U(t)),\hat{e}(t))\leq & \lambda(\|\hat{e}(t)\|^2+\|e(t)\|^2)+\mu(|\hat{e}(t)|^2+|e(t)|^2)\cr\noalign{\vskip4truemm}
&+L\big(\|\hat{e}(t)\|+\|e(t)\|\big)\|\hat{U}(t)-U(t)\|,
\end{aligned}
\end{equation*}
and
\begin{equation*}
L\big(\|\hat{e}(t)\|+\|e(t)\|\big)\|\hat{U}(t)-U(t)\|\leq 2\theta\big(\|\hat{e}(t)\|^2+\|e(t)\|^2\big)+\frac{L^2}{4\theta}\|\hat{U}(t)-U(t)\|^2. \end{equation*}
To sum up, we obtain
\begin{equation*}
\begin{aligned}
2(B(t,u(t))-B(t,U(t)),\hat{e}(t))\leq& (\lambda+2\theta)\big(\|\hat{e}(t)\|^2+\|e(t)\|^2\big)+3\mu|e(t)|^2 \cr\noalign{\vskip4truemm}
&+2\mu |\hat{U}(t)-U(t)|^2+\frac{L^2}{4\theta}\|\hat{U}(t)-U(t)\|^2
\end{aligned}
\end{equation*}
for any positive $\theta<\frac{1}{4}(1-\lambda)$.

Using  the Cauchy-Schwarz inequality leads to
\begin{equation*}
\begin{aligned}
\displaystyle \frac{d}{dt}|\hat{e}(t)|^2&-3\mu|\hat{e}(t)|^2+(1-\lambda-4\theta )\big(\|e(t)\|^2+\|\hat{e}(t)\|^2\big)\leq \big(1+\frac{L^2}{4\theta}\big)\|\hat{U}(t)-U(t)\|^2 \cr\noalign{\vskip4truemm} &+2\mu|\hat{U}(t)-U(t)|^2+\frac{1}{\theta}\|\varphi_1(-k_nA)R_b(t)\|_{\star}^2+\frac{1}{\theta}\|(\varphi_1(-k_nA)-I)R(t)\|_{\star}^2.
\end{aligned}
\end{equation*}
Under the condition that $\hat{e}(0)=0$, we  integrate the above  inequality from $0$ to $t \in I_n$:
\begin{equation}\label{semilinear upper}
\begin{aligned}
\displaystyle
|\hat{e}(t)|^2&+(1-\lambda-4\theta)\int_{0}^t e^{3\mu(t-s)}\big(\|e(s)\|^2+\|\hat{e}(s)\|^2\big)ds\cr\noalign{\vskip4truemm}
\displaystyle
\leq& \int_0^t e^{3\mu(t-s)}\Big[\big(1+\frac{L^2}{4\theta}\big)\|\hat{U}(s)-U(s)\|^2+2\mu|\hat{U}(s)-U(s)|^2\Big]ds\cr\noalign{\vskip4truemm} & +\sum\limits_{i=1}^{n-1}\int_{t^{i-1}}^{t^{i}} \frac{1}{\theta}\Big[\|\varphi_1(-k_iA)R_b(s)\|_{\star}^2+\|(\varphi_1(-k_iA)-I)R(s)\|_{\star}^2\Big]ds\\
& +\int_{t^{n-1}}^{t} \frac{1}{\theta}\Big[\|\varphi_1(-k_nA)R_b(s)\|_{\star}^2+\|(\varphi_1(-k_nA)-I)R(s)\|_{\star}^2\Big]ds.
\end{aligned}
\end{equation}
For the case of $\mu=0$, we have the a posteriori error estimate
\begin{equation}\label{semilinear upper2}
\begin{aligned}
\displaystyle
|\hat{e}(t)|^2&+(1-\lambda-4\theta)\int_{0}^t \big(\|e(s)\|^2+\|\hat{e}(s)\|^2\big)ds\cr\noalign{\vskip4truemm}
\displaystyle
\leq& \int_0^t \big(1+\frac{L^2}{4\theta}\big)\|\hat{U}(s)-U(s)\|^2ds +\sum\limits_{i=1}^{n-1}\int_{t^{i-1}}^{t^{i}} \frac{1}{\theta}\Big[\|\varphi_1(-k_iA)R_b(s)\|_{\star}^2 \cr\noalign{\vskip4truemm}
&+\|(\varphi_1(-k_iA)-I)R(s)\|_{\star}^2\Big]ds+\int_{t^{n-1}}^{t} \frac{1}{\theta}\Big[\|\varphi_1(-k_nA)R_b(s)\|_{\star}^2 \cr\noalign{\vskip4truemm}
&+\|(\varphi_1(-k_nA)-I)R(s)\|_{\star}^2\Big]ds.
\end{aligned}
\end{equation}
These imply that the error $e(t)$ can be bounded by the a posteriori quantities $\|\hat{U}(t)-U(t)\|^2$, $\|\varphi_1(-k_nA)R_b(t)\|_{\star}^2$ and $\|(\varphi_1(-k_nA)-I)R(t)\|_{\star}^2$. Furthermore,  assuming that $(1-\lambda-4\theta)>0$, then
\begin{equation}\label{nonlinear lower bound}
\frac{1-\lambda-4\theta}{2}\|\hat{U}(s)-U(s)\|^2\leq(1-\lambda-4\theta)\big(\|e(s)\|^2+\|\hat{e}(s)\|^2\big).
\end{equation}
Integrating the above inequality \eqref{nonlinear lower bound} from $0$ to $t$, we  can obtain the desired lower bound
\begin{equation}
\frac{1-\lambda-4\theta}{2} \int_0^t \|\hat{U}(s)-U(s)\|^2ds \leq (1-\lambda-4\theta) \int_0^t \big(\|e(s)\|^2+\|\hat{e}(s)\|^2\big)ds.
\end{equation}

\section{Numerical experiments}\label{sec7}
We have formulated  optimal order  a posteriori error estimates of the exponential midpoint method for linear  and semilinear parabolic equations. The theoretical results will be verified by the  numerical examples. Throughout the numerical experiments, the matrix-valued function
 $\varphi_1(-k_nA)$  is  evaluated by the  Krylov subspace method,
 which possesses the fast convergence. The details about Krylov subspace method can be
found in \cite{Hochbruck1997, Berland2007}.

 \begin{example}\label{example1}
In this example, we  consider the linear parabolic equation \cite{Hochbruck2005b}
\begin{equation}\label{problem1}
\begin{aligned}
u_t(x,t)&=u_{xx}(x,t)+f(x,t), \quad t\in [0,1], \quad x\in [0,1],\\
u(t,0)&=u(t,1)=0, \quad t \in [0,1]; \quad u(0,x)=x(1-x), \quad x\in [0,1],
\end{aligned}
\end{equation}
with the exact solution
$$u(x,t)=x(1-x)\exp(t).$$
 \end{example} 
 
Let $\Delta x=1/M$ and $x_i =i\Delta x$. Using the standard central finite difference to approximate the
space derivative $u_{xx}$ leads to
\begin{equation}
\begin{aligned}
u^{\prime}_i(t)&=\Delta x^{-2}[u_{i-1}(t)-2u_{i}(t)+u_{i+1}(t)]+f_{i}(t), \quad t>0,\\
u_i(t)&=i\Delta x(1-i\Delta x), \quad i=1,2,\ldots,M-1, \quad t=0,\\
u_{0}(t)&=u_{M}(t)=0, \quad t\geq0,
\end{aligned}
\end{equation}
where $u_i(t) \approx u(t,x_i)$, and $f_i$ stands for $f$ at $(t,x_i)$. 

The discrete maximum norm and the $L^2(0,T;V)$-norm in time of the error $e(t)$ are defined by:
\begin{equation*}
E_{\infty}:=\max\limits_{1\leq n\leq N} \Big(\sum\limits_{i=1}^M \Delta x |e_{i}^{n}|^2\Big)^{\frac{1}{2}}, \qquad  E_1:=\Big(\int_0^T \|e(s)\|^2ds\Big)^{\frac{1}{2}},
\end{equation*}
and the error at time $T$, $E_T=|e(T)|=|e(t^N)|.$
We  introduce the a posteriori quantities $\mathcal{E}_f$ and  $\zeta_{U}$  defined as
\begin{equation*}
\mathcal{E}_f:=\Big(\sum\limits_{n=1}^{N}\int_{t^{n-1}}^{t^n}\|\varphi_1(-k_nA)R_f(s)\|_{\star}^2ds\Big)^{\frac{1}{2}},
\end{equation*}
and
\begin{equation*}
\zeta_{U}=\Big(\sum\limits_{n=1}^{N}\int_{t^{n-1}}^{t^n}\|(\varphi_1(-k_nA)-I)R(s)\|_{\star}^2ds\Big)^{\frac{1}{2}}.
\end{equation*}
Taking $\mathcal{E}_U$ as the square root of $\varepsilon_U$ defined in \eqref{posteriori quantity2}, and we define  the effectivity indices
 $ei_L$ and $ei_U$   as
 \begin{equation*}
 ei_L:=\frac{lower \ estimator}{\frac{5}{3}Err_1},\qquad ei_U:=\frac{upper \ estimator}{Err_T+\frac{5}{3}Err_1},
 \end{equation*}
with $lower \ estimator:= 2\mathcal{E}_U/5$ and $upper \  estimator:= \mathcal{E}_U+6\mathcal{E}_f+6\zeta_{U}$, respectively. 
In what follows, all the integrals  from $t^{n-1}$ to $t^n$ are approximated by the Gauss-Legendre quadrature formula with three nodes which can exactly integrate polynomials of degree at most $4$.

Taking $M=100$,  We run the exponential midpoint rule with the uniform  time stepsize $k=1/N, \ N\in \mathbb{N}^{+}$. Table \ref{table1-1} indicates that the exact errors $E_{T}$ and $E_{\infty}$ has order two, however the order of exact error $E_1$ is around $1.75$. Table \ref{table1-2} presents that  $\varphi_1(-k_nA)$ has an impact on  the orders of the a  posteriori quantities $\mathcal{E}_U$, $\mathcal{E}_f$ and $\zeta_U$, and their orders  are optimal in theory, but  it has the order reduction phenomenon.
We also present the effectivity indices of the exponential midpoint rule in Table \ref{table1-3},  the effectivity indices of the lower estimator is asymptotically constant (around $0.2$) and the upper estimator is close to $2.9$.

\begin{table}
  \centering
\centering
\caption{The exact errors $E_T$, $E_{\infty}$ and $E_1$, and their convergence orders of exponential midpoint rule for Example \ref{example1}}\label{table1-1}
 \vskip -4mm
\renewcommand\arraystretch{1.2}
  \begin{tabular}{c|c|c|c|c|c|c}
    \hline
N& $E_{T}$&  order  & $E_{\infty}$ & order & $E_1$ &order\\
\hline
10&4.2546e-03  &	        & 4.2546e-03  &	            & 7.2489e-03	& \\
\hline
20&1.1009e-03  & 1.9504 & 1.1009e-03  & 1.9504	& 2.2030e-03	& 1.7183\\
\hline
40&2.8141e-04  & 1.9679	& 2.8141e-04  & 1.9679	& 6.5734e-04	& 1.7448\\
\hline
80&7.1380e-05  & 1.9791	& 7.1380e-05  & 1.9791	& 1.9526e-04	& 1.7512\\
\hline
160&1.8020e-05 & 1.9859 & 1.8020e-05  & 1.9859	& 5.7892e-05	& 1.7540\\
\hline
320&4.5351e-06 & 1.9904	& 4.5351e-06  & 1.9904	& 1.7131e-05	& 1.7568\\
\hline
  \end{tabular}
\end{table}

\begin{table}
\centering
\caption{ The a posteriori error quantities $\mathcal{E}_U$, $\mathcal{E}_f$ and  $\zeta_U$, and their convergence orders of exponential midpoint rule for   Example \ref{example1}} \label{table1-2} \vskip -4mm
\renewcommand\arraystretch{1.2}
\begin{tabular}{c|c|c|c|c|c|c}
    \hline
    N& $\mathcal{E}_{U}$&  order  & $\mathcal{E}_f$ & order & $\zeta_U$ &order\\
      \hline
10&6.1767e-03	&       	& 7.2112e-04	&        & 5.4989e-03	& \\
  \hline
20&1.9688e-03	& 1.6495	& 2.1784e-04	& 1.7270 & 1.6932e-03	& 1.6994 \\
  \hline
40&5.7947e-04	& 1.7645	& 6.0487e-05	& 1.8485 & 5.0883e-04	& 1.7345 \\
  \hline
80&1.6608e-04	& 1.8029	& 1.5988e-05	& 1.9197 & 1.5062e-04	& 1.7563 \\
  \hline
160&4.6912e-05	& 1.8238	& 4.1149e-06	& 1.9581 & 4.4408e-05	& 1.7620 \\
  \hline
320&1.3035e-05	& 1.8476	& 1.0443e-06	& 1.9783 & 1.3007e-05	& 1.7716\\     \hline
       \end{tabular}
\end{table}

\begin{table}[htbp]
\footnotesize
\centering
\caption{ Exponential midpoint rule for   Example \ref{example1}:
Lower and upper estimators of the error and effectivity indices}\label{table1-3} \vskip -4mm
\renewcommand\arraystretch{1.5}
\begin{tabular}{c|c|c|c|c|c|c}
\hline
    N& $\frac{2}{5}\mathcal{E}_U$ & $\frac{5}{3}Err_1$ & $Err_T+\frac{5}{3}Err_1$ & $\mathcal{E}_U+6\mathcal{E}_f+6\zeta_U$ & $ei_L$ & $ei_U$ \\[2pt]
      \hline
10&2.4707e-03	& 1.2082e-02	& 1.6336e-02	& 4.3497e-02	& 0.2045	& 2.6626 \\
    \hline
20&7.8752e-04	& 3.6717e-03	& 4.7726e-03	& 1.3435e-02	& 0.2145	& 2.8151 \\
  \hline
40&2.3179e-04	& 1.0956e-03	& 1.3770e-03	& 3.9954e-03	& 0.2116	& 2.9016 \\
  \hline
80&6.6432e-05	& 3.2543e-04	& 3.9681e-04	& 1.1657e-03	& 0.2041	& 2.9377 \\
  \hline
160&1.8765e-05	& 9.6487e-05	& 1.1451e-04	& 3.3805e-04	& 0.1945	& 2.9522 \\
  \hline
320&5.2140e-06	& 2.8552e-05	& 3.3087e-05	& 9.7343e-05	& 0.1826	& 2.9421 \\
\hline
       \end{tabular}
\end{table}

\begin{example}\label{example2} In this example, we consider  another  linear parabolic equation \eqref{problem1} with  $f=2\exp(-t)-x(1-x)\exp(-t)$ (see, e.g., \cite{Wang2021,Wang2022a}), the exact solution of this problem is
\begin{equation*}
u(x,t)=x(1-x)\exp(-t).
\end{equation*}
\end{example}

We take uniform time stepsize $N=10, 20, 40, 80, 160, 320$ with fixed spatial mesh size $M=100$, i.e., $\Delta x=1/M$, the exact errors $E_T$,  $E_{\infty}$,  $E_1$  and their orders are indicated in Table \ref{table2-1}.  From Table \ref{table2-2}, we observe that  the  orders of $\mathcal{E}_U$,  $\zeta_U$ are close to $1.8$ and  $1.75$, respectively, and the order of  $\mathcal{E}_f$ is around $1.9$. The effectivity indices $ei_L$ and $ei_U$ of the exponential midpoint rule are shown in Table \ref{table2-3}.
It seems that  the effectivity indices of  the lower and upper estimators are
asymptotically constant.
\begin{table}[htbp]
\footnotesize
\centering
\caption{The exact errors  $E_T$, $E_{\infty}$ and  $E_1$, and their convergence orders of exponential midpoint rule for  Example \ref{example2}}\label{table2-1} \vskip -4mm
\renewcommand\arraystretch{1.4}
\begin{tabular}{c|c|c|c|c|c|c}
\hline
    N& $E_{T}$&  order  & $E_{\infty}$ & order & $E_1$ &order\\
    \hline
10&5.2831e-04  &        & 9.9250e-04 &         & 2.1908e-03	&  \\
 \hline
20&1.3601e-04  & 1.9576	& 2.5919e-04 & 1.9370  & 6.8130e-04	& 1.6851 \\
 \hline
40&3.4791e-05  & 1.9670	& 6.6612e-05 & 1.9602  & 2.1019e-04	& 1.6967 \\
 \hline
80&8.8384e-06  & 1.9769	& 1.6990e-05 & 1.9711  & 6.4524e-05	& 1.7038 \\
 \hline
160&2.2345e-06 & 1.9838	& 4.3069e-06 & 1.9800  & 1.9662e-05	& 1.7144 \\
 \hline
320&5.6298e-07 & 1.9888	& 1.0872e-06 & 1.9861  & 5.9433e-06	& 1.7261 \\
\hline
       \end{tabular}
\end{table}
\begin{table}[htbp]
\footnotesize
\centering
\caption{ The a posteriori error quantities $\mathcal{E}_U$, $\mathcal{E}_f$ and $\zeta_U$, and their convergence orders of exponential midpoint rule for Example \ref{example2}}\label{table2-2} \vskip -4mm
\renewcommand\arraystretch{1.4}
\begin{tabular}{c|c|c|c|c|c|c}
\hline
    N& $\mathcal{E}_{U}$&  order  & $\mathcal{E}_f$ & order & $\zeta_U$ &order\\
  \hline
10&2.0272e-03  &        & 2.2813e-04 &		    & 1.9831e-03	 \\
\hline
20&6.1543e-04  & 1.7198	& 6.7828e-05 & 1.7499	& 6.2285e-04 & 1.6708 \\
\hline
40&1.8153e-04  & 1.7614	& 1.8701e-05 & 1.8587	& 1.8774e-04 & 1.7302 \\
\hline
80&5.3155e-05  & 1.7719	& 4.9279e-06 & 1.9241	& 5.5594e-05 & 1.7557 \\
\hline
160&1.5404e-05 & 1.7870	& 1.2667e-06 & 1.9599	& 1.6385e-05 & 1.7626 \\
\hline
320&4.3790e-06 & 1.8146	& 3.2134e-07 & 1.9790	& 4.7965e-06 & 1.7723 \\
\hline
       \end{tabular}
\end{table}

\begin{table}[htbp]
\footnotesize
\centering
\caption{ Exponential midpoint rule for  Example \ref{example2}:
Lower and upper estimators of the error and effectivity indices}\label{table2-3} \vskip -4mm
\renewcommand\arraystretch{1.5}
\begin{tabular}{c|c|c|c|c|c|c}
\hline
    N& $\frac{2}{5}\mathcal{E}_U$ & $\frac{5}{3}Err_1$ & $Err_T+\frac{5}{3}Err_1$ & $\mathcal{E}_U+6\mathcal{E}_f+6\zeta_U$ & $ei_L$ & $ei_U$ \\[2pt]
  \hline
10&8.1088e-04 & 3.6513e-03	& 4.1796e-03 & 1.5295e-02 & 0.2221	& 3.6593 \\
  \hline
20&2.4617e-04 & 1.1355e-03	& 1.2715e-03 & 4.7595e-03 & 0.2168	& 3.7432 \\
  \hline
40&7.2612e-05 & 3.5032e-04	& 3.8511e-04 & 1.4202e-03 & 0.2073	& 3.6877 \\
  \hline
80&2.1262e-05 & 1.0754e-04	& 1.1638e-04 & 4.1629e-04 & 0.1977	& 3.5770 \\
  \hline
160&6.1616e-06 & 3.2770e-05	& 3.5005e-05 & 1.2131e-04 & 0.1880	& 3.4657 \\
  \hline
320&1.7516e-06 & 9.9055e-06	& 1.0468e-05 & 3.5086e-05 & 0.1768	& 3.3516 \\
\hline
       \end{tabular}
\end{table}

\begin{example}\label{example3}
We consider the semilinear parabolic problem (see, e.g., \cite{Hochbruck2005a,Hochbruck2005b})
\begin{equation}\label{problem3}
\frac{\partial u}{\partial t}(x,t)-\frac{\partial ^2u}{\partial x^2}(x,t)=\frac{1}{1+u(x,t)^2}+\phi(x,t), \quad t\in[0,1], \quad x\in[0,1],
\end{equation}
with homogeneous Dirichlet boundary conditions. 
\end{example}
The $\phi(x,t)$ is chosen in such a  way that the exact solution is $u(x,t)=x(1-x)\exp(t)$. For the semilinear parabolic problem, it should be noted that  the standard exponential midpoint method is implicit and  iterations are required for the implementation of it. In this example, we use the fixed-point iteration, once the norm of the difference of two successive approximations is smaller than $10^{-10}$, then the iteration will be stopped.  Using the finite difference
to discretize the space derivative:
\begin{equation}
\begin{aligned}
u^{\prime}_i(t)&=\Delta x^{-2}[u_{i-1}(t)-2u_{i}(t)+u_{i+1}(t)]+\frac{1}{1+u_i(t)^2}+\phi_i(t), \quad t>0,\\
u_i(t)&=i\Delta x(1-i\Delta x), \quad i=1,2,\cdots,M-1, \quad t=0,\\
u_{0}(t)&=u_{M}(t)=0, \quad t\geq0,
\end{aligned}
\end{equation}
where $\Delta x=1/M, \ M=100$, $x_i=i\Delta x$, $u_i(t)\approx u(x_i,t)$, and $\phi_i(t)$ stands for $\phi$ at $(x_i,t)$.  As we stated in Section \ref{sec6}, the   upper  and lower error bounds  are  derived by  \eqref{semilinear upper2} and \eqref{nonlinear lower bound}, respectively. For this example, 
we choose $\theta=1/6$, $\lambda=1/6$, and $1+L^2/4\theta=2$, the a posteriori
quantities $\mathcal{E}_b$ and $\zeta_U$ are defined by
\begin{equation*}
\mathcal{E}_b:=\Big(\sum\limits_{n=1}^{N} \int_{t^{n-1}}^{t^{n}}\|\varphi_1(-k_nA)R_b(s)\|_{\star}^2ds\Big)^{\frac{1}{2}},
\end{equation*}
and
\begin{equation*}
\zeta_{U}=\Big(\sum\limits_{n=1}^{N} \int_{t^{n-1}}^{t^{n}} \|(\varphi_1(-k_nA)-I)R(s)\|_{\star}^2ds\Big)^{\frac{1}{2}}.
\end{equation*}
 The another a posteriori quantity  $\mathcal{E}_U$  is defined as
 \begin{equation*}
\mathcal{E}_U=\Big(\frac{1}{30}\sum\limits_{n=1}^N k_n^5\Big\|\frac{1}{k_n}\varphi_1(-k_nA)\big[B(t^{n-\frac{1}{2}},U^{n-\frac{1}{2}})-B(t^{n-1},U^{n-1})\big]\Big\|^2\Big)^{\frac{1}{2}}.
\end{equation*}
 Let us
define the effectivity indices $ei_L$ and $ei_U$ as:
 \begin{equation}
 ei_L:=\frac{lower \ estimator}{\frac{1}{3}Err_1},\qquad ei_U:=\frac{upper \ estimator}{Err_T+\frac{1}{3}Err_1}.
 \end{equation}
here $lower\ estimator := \mathcal{E}_U/12$ and $upper \ estimator:= 2\mathcal{E}_U+6\mathcal{E}_b+6\zeta_{U}$, respectively.

For this example, we take uniform time stepsize $k=1/N,\ N=10, 20, \ldots,360$. Table \ref{table3-1} indicates  the roughly second-order convergence rate of  exact errors $E_T$ and $E_{\infty}$, and the order of the error $E_1$ is around  $1.75$.  Table  \ref{table3-2} presents the order of the a posteriori quantities,  and the order reduction phenomenon of  $\mathcal{E}_U$ and $\zeta_U$ can be observed. Because the order of the exact error $E_1$ is around $1.75$, it is acceptable that the order of $\zeta_U$ is close to  $1.75$.   The effectivity indices $ei_L$ and $ei_U$ of the
exponential midpoint method are reported in Table \ref{table3-3}.

\begin{example}\label{example4}
Consider the Allen-Cahn equation \cite{Allen1979} 
\begin{equation}\label{semilinear problem2}
\begin{aligned}
  \frac{\partial u}{\partial t} (x,t)&=\epsilon \frac{\partial^2 u}{\partial x^2}+u-u^3, \qquad  \quad t \in [0,1], \quad x\in [-1,1], \\
  u(x,0)&=0.53x+0.47\sin(-1.5\pi x), \quad u(-1,t)=-1, \quad u(1,t)=1.
  \end{aligned}
\end{equation}
\end{example} 

For this example, the exact solution cannot be obtained, and  we take the exponential midpoint method with smaller time stepsize $k=1/10000$ as the reference solution. The space derivative is discretized  by the second-order finite difference, and taking the spatial mesh with $M = 80$. The uniform time stepsize is $k=1/N, N=10,20,40,80$. The exact errors  $E_T$, $E_{\infty}$, $E_1$ and their convergence orders of the exponential midpoint method are shown in Table \ref{table4-1}, and  the order of  $E_T$  is  around  $1.97$. Form Table \ref{table4-2},  we  observe that the entire function $\varphi_1(-k_nA)$ still affects the orders of a posteriori quantities $\mathcal{E}_U$, $\mathcal{E}_b$,  $\zeta_U$. 

\begin{table}[htbp]
\footnotesize
\centering
\caption{ The exact errors  $E_T$,  $E_{\infty}$ and $E_1$, and their convergence orders of exponential midpoint method for Example \ref{example3}}\label{table3-1} \vskip -4mm
\renewcommand\arraystretch{1.2}
\begin{tabular}{c|c|c|c|c|c|c}
\hline
    N& $E_{T}$&  order  & $E_{\infty}$ & order & $E_1$ &order\\
      \hline
10&4.3805e-03  &        & 4.3805e-03 &        & 7.3671e-03 & \\
  \hline
20&1.1327e-03  & 1.9513 & 1.1327e-03 & 1.9513 & 2.2405e-03 & 1.7172 \\
  \hline
40&2.8926e-04  & 1.9693	& 2.8926e-04 & 1.9693 & 6.6738e-04 & 1.7473 \\
  \hline
80&7.3312e-05  & 1.9803	& 7.3312e-05 & 1.9803 & 1.9768e-04 & 1.7554 \\
  \hline
160&1.8498e-05 & 1.9867	& 1.8498e-05 & 1.9867 & 5.8440e-05 & 1.7582 \\
  \hline
320&4.6536e-06 & 1.9909	& 4.6536e-06 & 1.9909 & 1.7252e-05 & 1.7602 \\
\hline
       \end{tabular}
\end{table}
\begin{table}[htbp]
\footnotesize
\centering
\caption{The a posteriori error quantities $\mathcal{E}_U$, $\mathcal{E}_b$ and $\zeta_U$, and their convergence orders of exponential midpoint method for  Example \ref{example3}}\label{table3-2} \vskip -4mm
\renewcommand\arraystretch{1.2}
\begin{tabular}{c|c|c|c|c|c|c}
\hline
    N& $\mathcal{E}_{U}$&  order  & $\mathcal{E}_b$ & order & $\zeta_U$ &order\\
    \hline
10&6.0946e-03  &	    & 1.1842e-03 &        & 5.6598e-03 &      \\
      \hline
20&1.9557e-03  & 1.6399 & 3.6621e-04 & 1.6932 & 1.7072e-03 & 1.7291 \\
  \hline
40&5.7762e-04  & 1.7595 & 1.0233e-04 & 1.8395 & 5.1015e-04 & 1.7427 \\
  \hline
80&1.6583e-04  & 1.8004 & 2.7086e-05 & 1.9176 & 1.5073e-04 & 1.7589 \\
  \hline
160&4.6882e-05 & 1.8227 & 6.9726e-06 & 1.9578 & 4.4417e-05 & 1.7628 \\
  \hline
320&1.3031e-05 & 1.8471	& 1.7695e-06 & 1.9784 & 1.3008e-05 & 1.7718 \\
\hline
       \end{tabular}
\end{table}

\begin{table}[htbp]
\footnotesize
\centering
\caption{ Exponential midpoint method for  Example \ref{example3}:
Lower and upper estimators of the error and effectivity indices}\label{table3-3} \vskip -4mm
\renewcommand\arraystretch{1.2}
\begin{tabular}{c|c|c|c|c|c|c}
\hline
    N& $\frac{1}{12}\mathcal{E}_U$ & $\frac{1}{3}Err_1$ & $Err_T+\frac{1}{3}Err_1$ & $2\mathcal{E}_U+6\mathcal{E}_b+6\zeta_U$ & $ei_L$ & $ei_U$ \\[2pt]
      \hline
10&5.0788e-04  & 2.4557e-03	& 6.8362e-03 & 5.3253e-02 & 0.2068	& 7.7899 \\
   \hline
20&1.6297e-04  & 7.4683e-04	& 1.8795e-03 & 1.6352e-02 & 0.2182	& 8.7002 \\
\hline
40&4.8135e-05  & 2.2246e-04	& 5.1172e-04 & 4.8301e-03 & 0.2164	& 9.4390 \\
\hline
80&1.3819e-05  & 6.5893e-05	& 1.3921e-04 & 1.3986e-03 & 0.2097	& 10.0467 \\
\hline
160&3.9068e-06 & 1.9480e-05	& 3.7978e-05 & 4.0210e-04 & 0.2006	& 10.5877 \\
\hline
320&1.0859e-06 & 5.7507e-06	& 1.0404e-05 & 1.1473e-04 & 0.1888	& 11.0275 \\
\hline
       \end{tabular}
\end{table}

 \begin{table}[htbp]
\footnotesize
\centering
\caption{The exact errors $E_T$, $E_{\infty}$ and  $E_1$, and their convergence orders of exponential midpoint method for  Example \ref{example4}}\label{table4-1} \vskip -4mm
\renewcommand\arraystretch{1.2}
\begin{tabular}{c|c|c|c|c|c|c}
\hline
    N& $E_{T}$&  order  & $E_{\infty}$ & order & $E_1$ &order\\
     \hline
10&1.8515e-04  &    	& 2.4752e-04  &          & 3.8948e-04  &  \\
 \hline
20&4.7053e-05 & 1.9763 & 9.3920e-05   &  1.3980  & 1.8163e-04  &  1.1005 \\
 \hline
40&1.1947e-05 & 1.9776 & 2.6510e-05   & 1.8249   & 7.4363e-05  &  1.2884\\
 \hline
80&3.0079e-06 & 1.9899 & 7.0526e-06   & 1.9103   & 2.1540e-05  &  1.7875 \\
\hline
       \end{tabular}
\end{table}

\begin{table}[htbp]
\footnotesize
\centering
\caption{The a posteriori error quantities $\mathcal{E}_U$, $\mathcal{E}_b$ and  $\zeta_U$, and their convergence orders of exponential midpoint method for  Example \ref{example4}}\label{table4-2} \vskip -4mm
\renewcommand\arraystretch{1.2}
\begin{tabular}{c|c|c|c|c|c|c}
\hline
    N& $\mathcal{E}_{U}$&  order  & $\mathcal{E}_b$ & order & $\zeta_U$ &order\\
  \hline
10& 6.8696e-04	&        &  2.2165e-03&        & 3.4676e-02 &       \\
  \hline
20&3.0548e-04	&  1.1691 & 7.9447e-04& 1.4802 & 1.3509e-02	&1.3600\\
  \hline
80&1.0873e-04	&  1.4904 & 3.0014e-04 &  1.4043& 4.4506e-03	& 1.6019 \\
  \hline
160&3.1158e-05	&   1.8030 & 9.6096e-05 & 1.6431 & 1.3032e-03	&   1.7720 \\
\hline
       \end{tabular}
\end{table}

\section{Concluding remarks}
In this paper, we studied the a posteriori error estimates  for linear and semilinear parabolic equations by using the exponential midpoint method in time. For linear problems, we introduced a continuous and piecewise linear $U(t)$ of the exponential midpoint rule to approximate $u(t)$, and derived the suboptimal order estimates.  In order to recover optimal order, we introduced a continuous and piecewise quadratic time reconstruction $\hat{U}$  of $U$ by using the property of the entire function. Then the lower and upper error bounds in the $L^2(0,T;V)$ and $L^{\infty}(0,T;H)$-norms were derived in detail which  depended only on the discrete solution and the approximation of the entire function. We also extended this technique to the semilinear problems, and established optimal order a posteriori error estimates.

We have implemented various numerical experiments for the exponential midpoint method for some linear and semilinear examples. For both cases, these experiments exactly verify the theoretical analysis. In this paper, we only present the a posteriori error estimates for the exponential midpoint method for parabolic problems, whether it is applicable to exponential Runge-Kutta methods of collocation type will be further investigated.

\section{Declarations}

\bmhead{Acknowledgements}
The authors are very grateful to the editor and anonymous referees for their invaluable comments and suggestions which helped to improve the manuscript.

\bmhead{Funding}
The first author was partially supported by the Natural Science Foundation of China (Grant No. 12271367, 12071419). The second author is the corresponding author. He was supported by the Natural Science Foundation of China (Grant No. 12271367), Shanghai Science and Technology Planning Projects (Grant No. 20JC1414200), and Natural Science Foundation of Shanghai (Grant No. 20ZR1441200).

\bmhead{Conflicts of interest}
The authors declare there is no conflicts of interest regarding the publication of this paper.

\bmhead{Ethical approval}
Not applicable.

\bmhead{Contributions}
All the authors contributed equally.

\bmhead{Availability of data and material}
All data, models, or code generated or used during the study are available from the corresponding author by request.

\bibliography{sn-bibliography}

\end{document}